\documentclass[11pt]{amsart}
\usepackage{amssymb}
\usepackage{amsrefs}

\newtheorem{Thm}{Theorem}[section]

\newtheorem{Cor}[Thm]{Corollary}
\newtheorem{Lem}[Thm]{Lemma}
\newtheorem{Prop}[Thm]{Proposition}

\def\ldots{\mathinner{\ldotp\ldotp\ldotp}}
\def\cdots{\mathinner{\cdotp\cdotp\cdotp}}

\def \Sp{\text{Sp}}

\begin{document}

\title[A new approach to Ramsey-type games]{A new approach to the Ramsey-type games and the Gowers dichotomy in F-spaces }

\author{G. Androulakis}
\address{Department of Mathematics \\
University of South Carolina \\
Columbia, SC 29208}

\email{giorgis@math.sc.edu}

\author{S. J. Dilworth}
\address{Department of Mathematics \\
University of South Carolina \\
Columbia, SC 29208}

\email{dilworth@math.sc.edu}

\author{N. J. Kalton}
\address{Department of Mathematics \\
University of Missouri-Columbia \\
Columbia, MO 65211 }

\email{nigel@math.missouri.edu}


\thanks{The second author was supported by NSF grant DMS-0701552. The third author was supported by NSF grant DMS-0555670}

\begin{abstract} We give a new approach to the Ramsey-type results of Gowers on block bases in Banach spaces and apply our results to prove the Gowers dichotomy in F-spaces.
\end{abstract}

\maketitle

\section{Introduction}  Our aim in this note is to establish the Gowers dichotomy \cite{Gowers1996} in a general F-space  (complete metric linear space).  We say that an F-space $X$ is {\it hereditarily indecomposable} if it is impossible to find two {\it separated} infinite-dimensional closed subspaces $V,W$, i.e. such that $V\cap W=\{0\}$ and $V+W$ is closed (or equivalently that the natural projection from $V+W$ onto $V$ is continuous).
Our main result is that an F-space either contains an unconditional basic sequence or an infinite-dimensional HI subspace.
 In order to prove such a result we give a new and, we hope, interesting approach to the Gowers Ramsey-type result about block bases in a Banach space.  We now state this result (terminology is explained in \S \ref{countable} and in \cite{Gowers2002}, \cite{Gowers2003}):

 \begin{Thm} [\cite{Gowers1996},\cite{Gowers2002}, \cite{Gowers2003}] \label{Th:G1}
Let $X$ be a Banach space with a basis. Let $\partial B_X$ denote the unit sphere of $X$ i.e.
$\partial B_X= \{ x \in X: \| x \| =1 \}$. Let $\sigma \subseteq \Sigma_{<\infty}(\partial B_X)$. Let $\Theta = (\theta_i)_i$ be a sequence
 of positive numbers. If $\sigma$ is large then there exists a block subspace $Y$ of $X$ such that
$\sigma_{\Theta}$ is strategically large for $Y$, where $\sigma_{\Theta}$ is the set of all finite block bases $\{u_1,\ldots,u_n\}$ such that for some $\{v_1,\ldots,v_n\}\in\sigma$ we have $\|u_i-v_i\|<\theta_i$.
\end{Thm}

In \cite{Gowers2002} and \cite{Gowers2003} the statement of Theorem~\ref{Th:G1} is announced for
$\partial B_X$ replaced by the unit ball except the origin, i.e. $B=B_X\setminus\{0\}=\{ x \in X: 0 < \| x \| \leq 1 \}$.
There appears to be a slight problem in the non-normalized case in
\cite{Gowers2002}*{page 805, line -9} and \cite{Gowers2003}*{page 1092, line -2}. Theorem~\ref{Th:G1} (including the non-normalized case)
follows from our Theorem~\ref{first}.

Gowers also considers an infinite version of the same result (Theorem 4.1 of \cite{Gowers2002}):

\begin{Thm} [\cite{Gowers2002}, \cite{Gowers2003}] \label{Th:G2}
Let $X$ be a Banach space with a basis.  Let $\sigma \subseteq \Sigma_{\infty}(\partial B_X)$. Let $\Theta = (\theta_i)_i$ be a sequence
 of positive numbers. If $\sigma$ is analytic and large then there exists a block subspace $Y$ of $X$ such that
$\sigma_{\Theta}$ is strategically large for $Y$, where $\sigma_{\Theta}$ is the set of all infinite block bases $\{u_1,\ldots,u_n,\ldots\}$ such that for some $\{v_1,\ldots,v_n,\ldots\}\in\sigma$ we have $\|u_i-v_i\|<\theta_i$.
\end{Thm}

Other proofs of these results can be found in the work of Bagaria and Lopez-Abad \cite{BagariaLopezAbad2001}, \cite{BagariaLopezAbad2002}.  Direct proofs of the dichotomy result without these theorems can be found in \cite{Maurey1999} and \cite{Figieletal2004}; see also \cite{Pelczar2003}.

 Our main objective is to prove Theorem \ref{Th:G2} in a form that is suitable for our intended applications.
  We take a somewhat different viewpoint (see Theorem \ref{main} below) by treating this theorem as a result about block bases in a countable dimensional space $E$ with no topology assumed.  We consider in fact only the intrinsic topology on $E$ i.e. the finest vector space topology. We then give a proof  which is rather distinct from that given by Gowers, and we feel has some advantages. A benefit of this approach is that we are able to apply the result very easily to the setting of a general F-space.

  In \S\ref{applications} we prove that the Gowers dichotomy extends to general F-spaces and discuss connections with similar (but easier) dichotomies for the existence of basic sequence.  In the final section, \S\ref{strictly} we prove the result of Gowers and Maurey \cite{GowersMaurey1993} that on a complex HI-space every operator is the sum of a scalar and a strictly singular operator in the context of quasi-Banach spaces.  This generalization is not entirely trivial and requires a few new tricks, although we broadly follow the same ideas as Gowers and Maurey.

\section{Countable dimensional vector spaces}\label{countable}

Let $E$ be a real or complex vector space of countable algebraic dimension.  There is a natural intrinsic topology $\mathcal T=\mathcal T_E$ on $E$ defined as follows:  a set $U$ is $\mathcal T-$open if $U\cap F$ is open relative to $F$ for every finite-dimensional subspace $F.$  The topology $\mathcal T$ is a vector topology on $E$ and is, indeed, the finest vector topology on $E$.  It is known that $(E,\mathcal T)$ is in fact locally convex.  More precisely if $(e_j)_{j=1}^{\infty}$ is any fixed Hamel basis then
the topology is induced by the family of norms
$$\|\sum_{j=1}^ma_je_j\|_{\lambda}=\max_{1\le j\le m}\lambda_j|a_j|$$ where $\lambda=(\lambda_j)_{j=1}^{\infty}$ is any sequence of positive numbers.  In the case when $\lambda_j=1$ for all $j$ we denote the resulting norm by $\|\cdot\|_\infty.$

We will also be concerned with the product $E^{\mathbb N}$.  On this there are two natural topologies: the product topology $\mathcal T_p$ and the box topology.  The box topology $\mathcal T_{bx}$ is a topology which makes $E^{\mathbb N}$ a topological group but not a topological vector space.  A base of neighborhoods of the origin is given by sets of the form $\prod_{n=1}^{\infty}U_n$ where each $U_n$ is a $\mathcal T-$neighborhood of zero in $E$.   A base can also be given by sets of the form
$\prod_{n=1}^{\infty}\{x:\|x\|_\lambda <\delta_n\}$ for some fixed norm $\|\cdot\|_\lambda $ and a sequence $\delta_n>0.$
We observe the obvious fact that if $V$ is an infinite-dimensional subspace of $E$ then $\mathcal T_E|V=\mathcal T_V$ and $\mathcal T_{E,bx}|V^{\mathbb N}=\mathcal T_{V,bx}.$

Now let us suppose that $E$ has a given fixed Hamel basis $(e_n)_{n=1}^{\infty}$.     Let $E_n=[e_1,\ldots,e_n]$ and $E^{(n)}=[e_{n+1},e_{n+2},\ldots].$      A sequence $(v_k)_{k=1}^n$ where $1\le n\le \infty$ is called a {\it block basis} of $(e_k)_{k=1}^{\infty}$ if each $v_k\neq 0$ and
$$ v_k=\sum_{j=p_{k-1}+1}^{p_k}a_je_j $$ for some increasing sequence $p_0=0<p_1<p_2<\cdots.$  A subspace $V$ of $E$ is called a block subspace if $V$ is the linear span of a block basis.

We let $\Sigma_{\infty}(E)$ be the subset of $E^{\mathbb N}$ consisting of all infinite block bases.  For each $n\in\mathbb N$ we let $\Sigma_n(E)$ be the subset of $E^{\mathbb N}$ of all block bases of length $n$.  We also let $\Sigma_0(E)$ be the one-point set with a single member $\emptyset.$  Let $\Sigma_{<\infty}(E)$ denote the union of all $\Sigma_n(E)$ for $0\le n<\infty.$
If $A$ is a subset of $E$ we denote by $\Sigma_n(A)$, etc., the subset of $\Sigma_n(E)$ with each element in $A$. In particular
we will be interested in the sets
$$ A_{\infty}=\{x\in E:\ 0<\|x\|_{\infty}\le 1\},\ S_{\infty}= \{x\in E:\ \|x\|_{\infty}=1\}.$$

\begin{Lem}\label{Polish} Let $\|\cdot\|$ be any norm on $E$ so that $(e_n)_{n=1}^{\infty}$ is a Schauder basis of the completion $\tilde E$ of $(E,\|\cdot\|).$ Then, on the space $\Sigma_{\infty}(E)$ the product topology $\mathcal T_p$ coincides with the product topology induced by $\|\cdot\|.$  In particular $(\Sigma_{\infty},\mathcal T_p)$ is a Polish space.
\end{Lem}

\begin{proof} Let $\xi_n=(\xi_{nk})_{k=1}^{\infty}$ be a sequence in $\Sigma_{\infty}(E)$ so that for some $\xi=(\xi_k)_{k=1}^{\infty}\in \Sigma_{\infty}(E),$ $\lim_{n\to\infty}\|\xi_{nk}-\xi_k\|=0$ for each $k$.
Let us suppose $\xi_{nk}\in [e_{p_{n,k-1}+1},\ldots,e_{p_{nk}}]$ where $p_{n0}<p_{n1}<\cdots$ and that $\xi_k\in [e_{p_{k-1}+1},\ldots,e_{p_k}].$  Then it is clear that
$$\limsup_{n\to\infty}p_{n,k-1} \le p_{k} \qquad k=1,2,\ldots$$
using the fact that $(e_n)_{n=1}^{\infty}$ is a Schauder basis.  It follows that each sequence $(\xi_{nk})_{n=1}^{\infty}$ is contained in some fixed finite-dimensional space and so the convergence is also in $\mathcal T_p.$

For the product-norm topology it is also easy to see that $\Sigma_{\infty}(E)$ is a closed subset of $(\tilde E\setminus \{0\})^{\mathbb N}$ and hence is Polish.\end{proof}

 Let $\mathcal B=\mathcal B(E)$ be the collection of all infinite-dimensional block subspaces of $(e_k)_{k=1}^{\infty}$. If $V\in\mathcal B$ then $V$ is the span of a block basis $(v_n)_{n=1}^{\infty}$ and we write $\mathcal B(V)$ for the collection of infinite-dimensional block subspaces of $V$ with respect to $(v_n)_{n=1}^{\infty}$ (this is clearly independent of the choice of the block basis). We will use the notation $(v_1,\ldots,v_r)\prec(u_1,\ldots,u_s)$ to mean that $(v_1,\ldots,v_r)$ is a block basis of $(u_1,\ldots,u_s).$

 Let $\sigma$ be a subset of $\Sigma_{\infty}(E).$  We shall say that $\sigma$ is {\it large} if for every $V\in\mathcal B(E)$ we have $\sigma \cap\Sigma_{\infty}(V)\neq \emptyset.$

 A  {\it strategy} is a map $\Phi:\Sigma_{<\infty}(E)\times \mathcal B(E)\to \Sigma_{< \infty}(E)$ if for all $(u_1,\ldots,u_n)\in \Sigma_n(E)$ we have $\Phi(u_1,u_2,\ldots,u_n;V)=(u_1,\ldots,u_n,u_{n+1})$ with $u_{n+1}\in V.$

 If $(V_j)_{j=1}^{\infty}$ is a sequence of block subspaces then we will write
 $$ \Phi(u_1,\ldots,u_n;V_1,\ldots,V_m)=(u_{1},\ldots,u_{m+n})$$ and
 $$ \Phi(u_1,\ldots,u_n;V_1,\ldots,V_m,\ldots )=(u_{1},\ldots,u_{m+n},\ldots)$$ where $u_{n+k}=\Phi(u_1,\ldots,u_{n+k-1};V_k)$ for $k\ge 1$.  In the case when $n=0$ we write
 $\Phi(V_1,\ldots,V_m)$ or $\Phi(V_1,\ldots,V_m,\ldots)$ for $\Phi(\emptyset;V_1,\ldots,V_m)$ or $\Phi(\emptyset;V_1,\ldots,V_m,\ldots)$

 A subset $\sigma$ of $\Sigma_{\infty}(E)$ is called strategically large for $V\in \mathcal B(E)$ and $(u_1,\ldots,u_n)\in \Sigma_{<\infty}(E)$ if there is a strategy $\Phi$ with the property that for every sequence  $(V_j)_{j=1}^{\infty}$ with $V_j\subset V$ we have $$\Phi(u_1,\ldots,u_n;V_1,\ldots,V_m,\ldots)\in \sigma.$$  $\sigma$ is strategically large for $V\in\mathcal B(E)$ if it is strategically large for $V\in\mathcal B(E)$ and $\emptyset.$

\section{Functions on subsets of $\Sigma_{<\infty}(E)$}

If $V,W$ are subspaces of $E$ let us write $V\subset_aW$ to mean that there exists a finite dimensional subspace $F$ so that $V\subset W+F.$

\begin{Lem}[Stabilization Lemma] \label{stabilization} Let $E$ be a countable dimensional space with fixed Hamel basis $(e_k)_{k=1}^{\infty}$.  Let $X$ be a separable topological space and suppose that, for each $V\in \mathcal B(E)$,  $f_V:X\to\mathbb R$ is a continuous function.  Suppose further that
$$ f_{V_1}(x)\ge f_{V_2}(x) \qquad x\in X$$ whenever $V_1\subset_a V_2$ .
Then there is a block subspace $W$ of $E$ so that
$f_V=f_W$ whenever $V\subset W.$

More generally suppose $(X_n)_{n=1}^{\infty}$ is a sequence of separable topological spaces and for each $V\in \mathcal B$ and $n\in\mathbb N,$  $f_V^{(n)}:X_n\to\mathbb R$ is a continuous function. Suppose further that
$$ f^{(n)}_{V_1}(x)\ge f^{(n)}_{V_2}(x) \qquad x\in X_n$$ whenever $V_1\subset_a V_2$.
Then there is a block subspace $W$ of $E$ so that
$f^{(n)}_V=f^{(n)}_W$ whenever $V\subset_a W$
and $n\in\mathbb N.$\end{Lem}

\begin{proof}  We prove the first part. We define block subspaces $V_{\alpha}$ for every countable ordinal $\alpha$ by transfinite induction, so that $\alpha\le \beta \ \implies V_{\beta}\subset_a V_{\alpha}.$  Set $V_1=E.$
For each $\alpha$ which is not a limit ordinal, say $\alpha=\beta+1$ define $V_{\alpha}\subset V_{\beta}$ so that $f_{V_{\alpha}}\neq f_{V_{\beta}}$ if possible; otherwise let $V_{\alpha}=V_{\beta}.$  If $\alpha$ is a limit ordinal then $\alpha=\sup_{n}\beta_n$ for some increasing sequence $(\beta_n)_{n=1}^{\infty}$ with $\beta_n<\alpha.$ Thus $V_{\beta_m}\subset_a V_{\beta_n}$ if $m>n.$  In this case we may by a diagonal argument find $V_{\alpha}$ so that $V_{\alpha}\subset_a V_{\beta_n}$ for every $n$ (simply choose a block basis $v_n$ with $v_n\in V_{\beta_1}\cap\cdots\cap V_{\beta_n}$). Now it follows that the functions $f_{V_{\alpha}}$ are increasing in $\alpha$ for $1\le \alpha<\omega_1.$  If $D$ is a countable dense set in $X$ there must therefore exist a countable ordinal $\beta$ so that
$$ f_{V_{\beta}}(x)=f_{V_{\alpha}}(x) \qquad x\in D,\ \beta\le \alpha.$$  Thus $f_{V_{\beta+1}}=f_{V_{\beta}}$ so that $W=V_{\beta}$ satisfies the conclusion.

The second part reduces to the first if we consider $X=\cup_{n=1}^{\infty}X_n$ topologized as a disjoint union and $f_V:X\to\mathbb R$ given by $f_V(x)=f^{(n)}_V(x)$ when $x\in X_n.$
\end{proof}

Consider a function $f:\Sigma_{<\infty}(A)\to [0,\infty)$ where
$A=S_{\infty}$ or $A=A_{\infty}.$ We shall say that $f$ is {\it
uniformly $\mathcal T_{bx}-$continuous} if given $\epsilon>0$
there is a sequence $(U_n)_{n=1}^{\infty}$ of $\mathcal
T-$neighborhoods of $0$ such that if
$(u_1,\ldots,u_r),(v_1,\ldots,v_r)\in \Sigma_{<\infty}(A)$ and
$u_j-v_j\in U_j$ for $1\le j\le r$ then
$$|f(u_1,\ldots,u_r)-f(v_1,\ldots,v_r)|<\epsilon.$$ In effect if
we induce maps $f^{[n]}$ on $\Sigma_{\infty}(A)$ by
$$ f^{[n]}(u_1,\ldots,u_k,\ldots)= f(u_1,\ldots,u_n)$$ this requires that the family of functions $(f^{[n]})_{n=1}^{\infty}$ is equi-uniformly continuous for the box topology $\mathcal T_{bx}.$

We will need a slightly weaker notion for maps $f:\Sigma_{<\infty}(A_{\infty})\to [0,\infty).$  We will say that $f$ is {\it admissible} if it is bounded and \newline (i) Given $\epsilon>0$, there is a sequence $(U_n)_{n=1}^{\infty}$ of $\mathcal T-$neighborhoods of $0$ such that if $(u_1,\ldots,u_r),(v_1,\ldots,v_r)\in \Sigma_{<\infty}(S_{\infty})$ and $u_j-v_j\in U_j$ for $1\le j\le r$ then $$|f(\lambda_1u_1,\ldots,\lambda_ru_r)-f(\lambda_1v_1,\ldots,\lambda_rv_r)|<\epsilon \qquad (\lambda_1,\ldots,\lambda_r)\in (0,1]^r,$$ and \newline
(ii) Given $\epsilon>0$ and $(u_1,\ldots,u_r)\in \Sigma_{<\infty}(S_{\infty})$ there exists $\delta=\delta(u_1,\ldots,u_r,\epsilon)>0$ so that if $0<\lambda_j,\mu_j\le 1$ for $1\le j\le r$ and $\max_{1\le j\le r}|\lambda_j-\mu_j|\le \delta$ then
$$ |f(\lambda_1u_1,\ldots,\lambda_ru_r,v_1,\ldots,v_s)-f(\mu_1u_1,\ldots,\mu_ru_r,v_1,\ldots,v_s)|<\epsilon,$$ whenever $ (u_1,\ldots,u_r,v_1,\ldots,v_s)\in \Sigma_{<\infty}(A_{\infty}).$

The following Lemma is easy and its proof is omitted:

\begin{Lem}\label{easy} (i) Suppose $f:\Sigma_{<\infty}(A_{\infty})\to [0,\infty)$ is bounded and uniformly $\mathcal T_{bx}$-continuous; then $f$ is admissible.
\newline
(ii) Suppose $f:\Sigma_{<\infty}(S_{\infty})\to [0,\infty)$  is uniformly $\mathcal T_{bx}$-continuous; then $g:\Sigma_{<\infty}(A_{\infty})\to [0,\infty)$ is admissible where
$g(u_1,\ldots,u_n)=f(u_1/\|u_1\|_{\infty},\ldots,u_n/\|u_n\|_{\infty}).$\end{Lem}

\begin{Lem}\label{continuous} If $f:\Sigma_{<\infty}(A_{\infty})\to [0,\infty)$ is admissible then for each $m\in\mathbb N$ the map $F_m:(0,1]^m\times \Sigma_m(S_{\infty})\to [0,\infty)$ defined by
$$ F(\lambda_1,\ldots,\lambda_m,u_1,\ldots,u_m)=f(\lambda_1u_1,\ldots,\lambda_mu_m)$$ is continuous when $\Sigma_m(S_{\infty})\subset (E,\mathcal T)^m$ is given the subset topology.
\end{Lem}

\begin{proof}  Suppose $\epsilon>0.$  We pick  $\mathcal T-$neighborhoods of zero in $E$, $U_1,\ldots,U_m$ so that
$u_j-v_j\in U_j$ for $1\le j\le m$ implies that
$$ |f(\lambda_1v_1,\ldots,\lambda_mv_m)-f(\lambda_1u_1,\ldots,\lambda_mu_m)|<\epsilon/2$$ for every $(\lambda_1,\ldots,\lambda_m)\in (0,1]^m.$ If $(v_1,\ldots,v_m)\in \Sigma_{m}(E_n\cap S_{\infty})$  we then pick $\delta=\delta(v_1,\ldots,v_r)>0$ so that if $|\lambda_j-\mu_j|<\delta$ for $1\le j\le m$ we have
$$ |f(\lambda_1v_1,\ldots, \lambda_mv_m)-f(\mu_1v_1,\ldots,\mu_mv_m)|<\epsilon/2.$$  Combining gives
$$ |f(\lambda_1u_1,\ldots,\lambda_mu_m)-f(\mu_1v_1,\ldots,\mu_mv_m)|<\epsilon$$ whenever $\max_{1\le j\le m}|\lambda_j-\mu_j|<\delta$ and $u_j-v_j\in U_j$ for $1\le j\le m.$ Thus $F$ is continuous at each point $(\mu_1,\ldots,\mu_r,v_1,\ldots,v_m)$.
\end{proof}

Suppose $f:\Sigma_{<\infty}(A_{\infty})\to [0,\infty)$ is any admissible function.
Let us adopt the convention that the function $f$ takes the value $+\infty$ at any point of $E^{\mathbb N}\setminus \Sigma_{<\infty}(A_{\infty}).$
For any $V \in {\mathcal B}(E)$ let \begin{align*} &f'_V(u_1,\ldots,u_n)=\\ &\lim_{m\to\infty} \inf\Big\{ f(u_1,\ldots,u_n,v_1,\ldots,v_s):\ v_1,\ldots,v_s\in V\cap E^{(m)},\ s\ge 1\Big\}.\end{align*}

Note that $V\subset_a W$ implies that $f'_V\ge f'_W.$

The following is more or less immediate:

\begin{Lem}\label{stab2}  If $f:\Sigma_{<\infty}(A_{\infty})\to [0,\infty)$ is admissible,   then the each of the functions $f'_V:\Sigma_{<\infty}(A_\infty)\to [0,\infty)$  is admissible.
\end{Lem}

\begin{Lem}\label{stab3} If $\mathcal F$ is a countable family of admissible functions,  then there exists $V\in\mathcal B(E)$ so that for every $W\in\mathcal B(V)$ and every $f\in\mathcal F$ we have $ f'_W= f'_V.$
\end{Lem}

\begin{proof} For $W\in\mathcal B(E)$ and $m<n$ define
$g_{m,n,W}:(0,1]^m\times \Sigma_{m}(S_{\infty}\cap E^n)\to\mathbb R$ by
$$ g_{m,,n,W}(\lambda_1,\ldots,\lambda_m,u_1,\ldots,u_m)=f'_W(\lambda_1u_1,\ldots,\lambda_mu_m).$$    Thus $g_{m,,n,W}$ is continuous by Lemma \ref{continuous}.  Since $(0,1]^m\times \Sigma_m(S_{\infty}\cap E^n)$ is separable for each $m$, $n$, we can
 apply the Stabilization Lemma \ref{stabilization}.\end{proof}

We can thus assume, under the hypotheses of the Lemma, (by passing to a block subspace) that $f$ has the property that $f'_V=f'_E$ for all block subspaces $V$.  If this happens we shall say that $f$ is {\it stable} and we write $f'$ for $f'_E.$  Note that $f'$ is admissible.

\begin{Prop}\label{admissible}  Let $f$ be a stable admissible function.  Suppose $(u_1,\ldots,u_r)\in \Sigma_{<\infty}(A_{\infty})$ and $V$ is a block subspace.  Then for any $\epsilon>0$ there exists $\xi\in V\setminus\{0\}$ so that  either:\newline
(a) $f(u_1,\ldots,u_r,\xi)<f'(u_1,\ldots,u_r)+\epsilon $ or \newline
(b) $f'(u_1,\ldots,u_r,\xi)<f'(u_1,\ldots,u_r)+\epsilon.$
\end{Prop}

\begin{proof}  Let us assume that $(u_1,\ldots,u_r)\in \Sigma_r(E).$  Let us further assume that $V$ is a block subspace so that
for any $\xi\in V$ we have
\begin{equation}\label{ineq} f(u_1,\ldots,u_r,\xi)\ge f'(u_1,\ldots,u_r)+\epsilon,\qquad f'(u_1,\ldots,u_r,\xi)\ge f'(u_1,\ldots,u_r)+\epsilon.\end{equation}

Let $(v_j)_{j=1}^{\infty}$ be a block basis which is a basis of $V.$
We choose an increasing sequence of integers $(q_k)_{k=1}^{\infty}$ as follows.  Let $q_1=1.$  Assume $q_1,\ldots,q_k$ have been chosen.  Let $m_0$ be the smallest integer so that  $v_{q_k}\in E_{m_0}.$
Then for every $\xi\in S_{\infty}\cap [v_{q_1},\ldots,v_{q_k}]$ \eqref{ineq} holds.
We may pick a neighborhood $U$ of $0$ in $(E,\mathcal T)$ so that
$$ |f(u_1,\ldots,u_r,\lambda \xi,w_1\ldots,w_s)-f(u_1,\ldots,u_r,\lambda\eta,w_1,\ldots,w_s)|<\epsilon/8$$ when $\xi-\eta\in U$, $(w_1,\ldots,w_s)\in \Sigma_{<\infty}(A_{\infty}\cap E^{(m_0)})$ and $0<\lambda\le 1.$  By compactness there is a finite subset $(\xi_1,\ldots,\xi_t)$ of
$S_{\infty}\cap[v_{q_1},\ldots,v_{q_k}]$ such that $\eta\in S_{\infty}\cap[v_{q_1},\ldots,v_{q_k}]$ implies $\eta-\xi_j\in U$ for some $j.$   Now pick an integer $N$ large enough so that $|\lambda-\mu|<1/N$ implies
$$ |f(u_1,\ldots,u_r,\lambda\xi_j,w_1,\ldots,w_s)-f(u_1,\ldots,u_r,\mu\xi_j,w_1,\ldots,w_s)|<\epsilon/8$$ whenever $(w_1,\ldots,w_s)\in \Sigma_{<\infty}(E^{(m_0)}).$
Now by our assumptions we can pick $m\ge m_0$ so that
$$ f(u_1,\ldots,u_r,\tfrac{k}{N}\xi_j,w_1,\ldots,w_s)>f'(u_1,\ldots,u_r)+\tfrac34\epsilon$$ whenever $1\le j\le t,1\le k\le N$ and $(w_1,\ldots,w_s)\in \Sigma_{<\infty}(E^{(m)}).$
Hence
$$ f(u_1,\ldots,u_r,\lambda\xi_j,w_1,\ldots,w_s)>f'(u_1,\ldots,u_r)+\tfrac58\epsilon$$ whenever $1\le j\le t, 0<\lambda\le 1$ and $(w_1,\ldots,w_s)\in \Sigma_{<\infty}(E^{(m)})$, and thus
\begin{equation}\label{ineq3} f(u_1,\ldots,u_r,\lambda\xi,w_1,\ldots,w_s)>f'(u_1,\ldots,u_r)+\tfrac12\epsilon\end{equation}
whenever $0< \lambda\le 1$ and $\xi\in S_{\infty}\cap[v_{q_1},\ldots,v_{q_k}].$

Then we pick $q_{k+1}>q_k$ so that $v_{q_{k+1}}\in E^{(m)}.$  This completes the inductive construction.
Now let $W=[v_{q_1},v_{q_2},\ldots].$   There exists $(w_1,\ldots,w_s)\in\Sigma_s(W)$ (where $s>0$) so that
$f(u_1,\ldots,u_r,w_1,\ldots,w_s)<f'(u_1,\ldots,u_r)+\epsilon/2.$

If $s=1$ then $\xi=w_1$ contradicts \eqref{ineq}.  If $s>1$ let $\lambda\xi=w_1=\sum_{j=1}^la_jv_{q_j}$ where $a_l\neq 0.$  Then by the selection of $q_{l+1}$ and \eqref{ineq3}  we see that
$f(u_1,\ldots,u_r,w_1,\ldots,w_s)>f'(u_1,\ldots,u_r)+\epsilon/2$, a contradiction.
\end{proof}

\begin{Lem}  If $f:\Sigma_{<\infty}(A_{\infty})\to [0,\infty)$ is admissible then the function $g:\Sigma_{<\infty}(A_{\infty})\to [0,\infty)$ given by
$g(u_1,\ldots,u_r)=1$ if $r=0$ or $r=1$ and
$$ g(u_1,\ldots,u_r)=\inf\{f(v_1,\ldots,v_s):\ (v_1,\ldots,v_s)\prec (u_1,\ldots,u_r),\ 1\le s<r,\ \} \qquad (r>1)$$
is admissible.\end{Lem}

\begin{proof} Note that $g$ satisfies $g(\lambda_1u_1,\ldots,\lambda_mu_m)=g(u_1,\ldots,u_m)$ if $(u_1,\ldots,u_m)\in \Sigma_{<\infty}(S_{\infty})$
and $(\lambda_1, \ldots , \lambda_m) \in (0,1]^m$.   Hence it
suffices to show that $g$ is uniformly $\mathcal
T_{bx}-$continuous on $\Sigma_{<\infty}(S_{\infty}).$  Note that,
 if for  all$(u_1,\ldots,u_m)\in \Sigma_{<\infty}(S_{\infty})$ we
define
$$ h(u_1,\ldots,u_m)=\inf\{ f(\lambda_1u_1,\ldots,\lambda_mu_m):\ (\lambda_1,\ldots,\lambda_m)\in (0,1]^m\},$$
 then $h$ is uniformly $\mathcal T_{bx}-$continuous and
$$ g(u_1,\ldots,u_r)=\inf\{h(v_1,\ldots,v_s):\ (v_1,\ldots,v_s)\prec (u_1,\ldots,u_r),\ 1\le s<r,\ \} \qquad (r>1)$$
 Suppose $\epsilon>0.$  Then there is a sequence $(U_n)_{n=1}^{\infty}$ of $\mathcal T-$neighborhoods of zero so that if $(u_1,\ldots,u_n),(v_1,\ldots,v_n)\in \Sigma_{<\infty}(S_{\infty})$ and $u_j-v_j\in U_j$ for $1\le j\le n$ then
$$ |h(u_1,\ldots,u_n)-h(v_1,\ldots,v_n)|<\epsilon.$$
Pick a sequence of circled neighborhoods of zero, $(U'_n)_{n=1}^{\infty},$ so that
$U'_n+U'_{n+1}+\cdots+U'_{n+k}\subset U_n$ whenever $k\ge n.$  Then suppose $(u_1,\ldots,u_n),(v_1,\ldots,v_n)\in \Sigma_{<\infty}(S_{\infty})$ with $u_j-v_j\in U'_j$ for $1\le j\le n.$  Assume $\eta>0$ and pick $(x_1,\ldots,x_r)\prec (u_1,\ldots,u_n)$ with $r<n$ so that
$h(x_1,\ldots,x_r)<g(u_1,\ldots,u_n)+\eta.$  If
$x_j=\sum_{i=k+1}^la_iu_i$ let
$y_j=\sum_{i=k+1}^la_iv_i.$  Then $|a_i|\le 1$ and so
$x_j-y_j\in U'_{k+1}+\cdots+U'_l\subset U_{j}$ since $j\le k+1.$  Hence
$$ h(y_1,\ldots,y_r)<g(u_1,\ldots,u_n)+\epsilon+\eta$$ and so
$$ g(v_1,\ldots,v_n)\leq g(u_1,\ldots,u_n)+\epsilon.$$  By symmetry
$$ |g(v_1,\ldots,v_n)-g(u_1,\ldots,u_n)|\le \epsilon$$ and hence $g$ is uniformly $\mathcal T_{bx}-$continuous.\end{proof}

We shall say that a strategy $\Phi$ is $(\epsilon,V)$-effective for $f$ where $\epsilon>0$ and $V\in\mathcal B(E)$ if
for every sequence of block subspaces $V_j\subset V$ there exists $n\in\mathbb N$ so that
$$ f(\Phi(\emptyset,V_1,\ldots,V_n))<\sup_{W\in \mathcal B(E)}f'_W(\emptyset)+\epsilon.$$  If $(u_1,\ldots,u_r)\in \Sigma_{<\infty}(E)$ we shall say that a strategy $\Phi$ is $(\epsilon,u_1,\ldots,u_r,V)$-effective for $f$ where $\epsilon>0$ and $V\in\mathcal B(E)$ if
for every sequence of block subspaces $V_j\subset V$ there exists $n\in\mathbb N$ so that
$$ f(\Phi(u_1,\ldots,u_r,V_1,\ldots,V_n))<\sup_{W\in\mathcal B(E)} f'_W(u_1,\ldots,u_r)+\epsilon.$$

\begin{Thm}\label{first}  Suppose $f:\Sigma_{<\infty}(A_{\infty})\to [0,\infty)$ is admissible.  Then, given $\epsilon>0$ there is a block subspace $V$ and a strategy $\Phi$ which is $(\epsilon,V)-$effective for $f$.
 \end{Thm}

\begin{proof}  We assume (after stabilization) that $f'_V(\emptyset)=0$ for all $V\in\mathcal B(E)$; indeed if $a=\sup_{V\in\mathcal B(E)}f'_V(\emptyset)$ then replace $f$ by $\max(f-a,0).$
  We consider the admissible function $$h(u_1,\ldots,u_r)=f(u_1,\ldots,u_r)+2(\epsilon-2g(u_1,\ldots,u_r))_+.$$  By Lemma \ref{stab2} we can pass to a block subspace where $h$ is stable; so let us assume $h$ is stable on $E$.

  We first claim that if $(u_1,\ldots,u_n,\ldots)\in \Sigma_{\infty}(A_\infty)$ then there exists $n$ so that $h'(u_1,\ldots,u_n)>\epsilon.$  Indeed, let $W=[u_1,\ldots,u_n,\ldots].$  Then, since $f'_W(\emptyset)=0,$  there exists $(v_1,\ldots,v_s)\in \Sigma_{<\infty}(W)$ with $f(v_1,\ldots,v_s)<\epsilon/4.$  Then we may find $r>s$ so that $(v_1,\ldots v_s)\prec (u_1,\ldots,u_r).$  Hence for any $(x_1,\ldots,x_t)$ such that $(u_1,\ldots,u_r,x_1,\ldots,x_t)\in \Sigma_{<\infty}(A_{\infty})$ we have $$h(u_1,\ldots,u_r,x_1,\ldots,x_t)\ge 2(\epsilon -2f(v_1,\ldots,v_s))$$ so that
  $$ h'(u_1,\ldots,u_r)\ge 2(\epsilon -2f(v_1,\ldots,v_s))>\epsilon,$$ which proves the claim.

  On the other hand, given any block subspace $V$, there exists a minimal $s \geq 1$ so that we can find $(v_1,\ldots,v_s)\in \Sigma_{<\infty}(V)$ with $f(v_1,\ldots,v_s)<\epsilon/2.$    Thus $g(v_1,\ldots,v_s)\ge \epsilon/2$ which implies $h(v_1,\ldots,v_s)<\epsilon/2.$  Hence $h'(\emptyset)<\epsilon/2.$

    We now use a strategy for $h$ indicated by Proposition \ref{admissible}. Suppose $\epsilon_j>0$ for each $j\ge 0$ and $\sum\epsilon_r<\epsilon/2.$  Given $(u_1,\ldots,u_r)\in \Sigma_{<\infty}(E)$ and $V\in\mathcal B(E)$ we define $\Phi(u_1,\ldots,u_r,V)$ to be $(u_1,\ldots,u_r,\xi)$ where $\xi\in V\setminus\{0\}$ is chosen so that $h'(u_1,\ldots,u_r,\xi)<h'(u_1,\ldots,u_r)+\epsilon_r$ or $h(u_1,\ldots,u_r,\xi)<h'(u_1,\ldots,u_r)+\epsilon_r.$    Let $(V_j)_{j=1}^{\infty}$ be a sequence in $\mathcal B(E)$ and let $(u_1,\ldots,u_r,\ldots)=\Phi(\emptyset;V_1,\ldots,).$  Then, by our previous claim there exists a first $n\ge 1$ so that  $h'(u_1,\ldots,u_n)>h'(u_1,\ldots,u_{n-1})+\epsilon_{n-1}.$   Hence $$h(u_1,\ldots,u_n)<h'(u_1,\ldots,u_{n-1})+\epsilon_{n-1}<h'(\emptyset)+\sum_{j=0}^{n-1}\epsilon_j<\epsilon.$$
\end{proof}

We need an obvious extension of this result.

\begin{Thm}\label{first2}  Suppose $f:\Sigma_{<\infty}(A_{\infty})\to [0,\infty)$ is admissible.  Then
there is a block subspace $V$ such that for each $(u_1,\ldots,u_r)\in \Sigma_{<\infty}( A_{\infty})$ there is a strategy $\Phi_{u_1,\ldots,u_r}$ which is $(\epsilon,u_1,\ldots,u_r,V)-$effective for $f$.\end{Thm}

\begin{proof}  It is easy to obtain, for each $(u_1,\ldots,u_r)$ the existence of a block subspace $V_{u_1,\ldots,u_r}$ and a strategy $\Psi_{u_1,\ldots,u_r}$ which is $(\epsilon/2,u_1,\ldots,u_r,V_{u_1,\ldots,u_r})-$effective for $f$.  Indeed, suppose $u_r\in E_m$; define
$$ f_1(v_1,\ldots,v_s)= f(u_1,\ldots,u_r,v_1,\ldots,v_s)\qquad (v_1,\ldots,v_s)\in \Sigma_{<\infty}(A_{\infty}\cap E^{(m)})$$  and apply the preceding theorem to $f_1$. ($\Psi_{u_1,\ldots,u_r}$ has to be defined in some arbitrary fashion for $(w_1,\ldots,w_s)$ which do not have $(u_1,\ldots,u_r)$ as an initial segment.)  Furthermore it can be seen that for each block subspace $W$ we can choose $V_{u_1,\ldots,u_r}\subset W$.

To obtain a single block subspace $V$ we first construct a dense
countable subset $D_{m,r}$ in each $\Sigma_{r}(E_m\cap\mathcal
A_\infty).$  We arrange the elements of $D=\cup_{m,r}D_{m,r}$ as a
sequence and hence find a descending sequence of subspaces $(V_n)$
so that the strategy $\Phi_{u_1,\ldots,u_r}$ is
$(\epsilon/2,u_1,\ldots,u_r,V_n)-$effective for $f$ when
$(u_1,\dots, u_r)$ is the $n$th. member of $D.$  If we select $V$
to be block subspace so that $V\subset V_n+F_n$ for some
finite-dimensional $F_n$ for each $n$, then (via a simple
modification) each $\Phi_{u_1,\ldots,u_r}$ is
$(\epsilon/2,u_1,\ldots,u_r,V)-$effective for $f$.   Finally we
observe that if $(v_1,\ldots,v_r)\in \Sigma_{r}(E_m)$ is arbitrary
and we then choose $(u_1,\ldots,u_r)\in D$ close enough, we can
define a strategy by $$
\Phi_{v_1,\ldots,v_r}(v_1,\ldots,v_r,w_1,\ldots,w_s)=\Psi_{u_1,\ldots,u_r}(u_1,\ldots,u_r,w_1,\ldots,w_s)
$$ (and arbitrarily otherwise) then we will have that
$\Phi_{v_1,\ldots,v_n}$ is $(\epsilon,v_1,\ldots,v_r,V)-$effective
for $f$.\end{proof}

\section{The infinite case}

We now turn to the infinite case. Suppose $f:\Sigma_{\infty}(A_{\infty})\to [0,\infty)$ is a bounded uniformly $\mathcal T_{bx}$-continuous
function.  We may define $f'_V:\Sigma_{<\infty}(A_{\infty})\to [0,\infty)$ in a precisely analogous way.  As before we adopt the convention that $f=+\infty$ on $E^{\mathbb N}\setminus \Sigma_{\infty}(A_{\infty}).$
Let
\begin{align*} &f'_V(u_1,\ldots,u_r)=\\ &\lim_{m\to\infty}\inf \Big\{ f(u_1,\ldots,u_r,v_1,\ldots,v_s,\ldots):\ v_j\in V\cap E^{(m)},\ j=1,2,\ldots\Big\}.\end{align*}
It is clear that the functions $\{f'_V:\ V\in\mathcal B(E)\}$ are equi-uniformly $\mathcal T_{bx}-$continuous.

Proceeding in the same manner as before we can show:

\begin{Prop}  If $f_n:\Sigma_\infty(A_\infty)\to \mathbb R$ is any countable family of bounded $\mathcal T_{bx}$ uniformly continuous
functions, there is a block subspace $V$ of $E$ so that $f'_{W}=f'_{V}$ whenever $W\in \mathcal B(V).$\end{Prop}

We shall say that $f$ is {\it stable} if $f'_E=f'_V$ for every $V\in\mathcal B(E).$

\begin{Lem}\label{sequences}  Let $f_n:\Sigma_{\infty}(A_{\infty})\to [0,\infty)$ be a sequence of bounded uniformly $\mathcal T_{bx}$-continuous functions and suppose $f=\inf f_n$ is also $\mathcal T_{bx}-$uniformly continuous.  Assume that each $f_n$ and $f$ are stable.  Let $h:\Sigma_{<\infty}(A_{\infty})=\inf_n f'_n.$ Then for every $V\in\mathcal B(E)$ we have $h'_V \le f'.$
\end{Lem}

\begin{proof}    Let $V\in\mathcal B(E).$ Let us assume $h_V'(u_1,\ldots,u_r)>\lambda>f'(u_1,\ldots,u_r)$ for some $(u_1,\ldots,u_r)\in \Sigma_{<\infty}(A_\infty \cap V).$
Then there exists $m$ so that if $(v_1,\ldots,v_s)\in \Sigma_{<\infty}(E^{(m)}\cap V)$ with $s\ge 1$
we have
$$h(u_1,\ldots,u_r,v_1,\ldots,v_s)>\lambda.$$  Thus $$f'_n(u_1,\ldots,u_r,v_1,\ldots,v_s)>\lambda,\qquad n=1,2\ldots.$$

Let us pick $w_1\in A_{\infty}\cap E^{(m)}\cap V.$  We will construct a sequence $(w_n)_{n=1}^{\infty}$ in $V$ by induction.  Suppose $(w_1,\ldots,w_n)$ have been selected and let $W_n=[w_1,\ldots,w_n].$  Then by compactness we can find $p$ so that $W_n\subset E_p$ and if $1\le k\le n$ and $(v_1,\ldots,v_k)\in \Sigma_k(W_n\cap A_{\infty})$ then
$$ f_k(u_1,\ldots,u_r,v_1,\ldots,v_k,x_1,\ldots)>\lambda$$ for all choices of $(x_j)_{j=1}^{\infty}$ in $E^{(p)}.$  Pick $w_{n+1}\in E^{(p)}\cap V.$

Now let $W=[w_j]_{j=1}^{\infty}.$  Our construction guarantees that for every $n$ it is true that
$$f_n(u_1,\ldots,u_r,x_1,\ldots)>\lambda \qquad x_j\in W,\ j=1,2,\ldots$$ Indeed we have $(x_1,\ldots,x_n)\in \Sigma_{n}([w_1,\ldots,w_m])$ where $m>n$ and so this follows from our inductive construction.
Thus $$f'(u_1,\ldots,u_r)\ge \lambda,$$ contradicting our initial hypothesis.
\end{proof}

We now use the space ${\mathbb N}^{\mathbb N}$ with the usual product topology; this can be regarded as the space of all infinite words of the natural numbers.  We will write $\mathbb N^{<\infty}=\cup_{n=0}^{\infty}\mathbb N^k$ which is the space of all finite words of the natural numbers, including the empty word.  We will use $(n_1,\ldots,n_k)$ or $(n_1,n_2 ,\ldots)$ to denote a typical member of $\mathbb N^{<\infty}$ or $\mathbb N^{\mathbb N}.$

\begin{Thm}\label{main0}  Suppose $F:{\mathbb N}^{\mathbb N}\times \Sigma_{\infty}(A_{\infty})\to [0,\infty)$ is a bounded map.  Define
$ f_{n_1,n_2,\ldots}:\Sigma_{\infty}(A_{\infty})\to [0,\infty)$ by
$$ f_{n_1,n_2,\ldots}(u_1,u_2,\ldots)= F(n_1,n_2,\ldots;u_1,u_2,\ldots).$$
Suppose
\newline (i) The maps $\{f_{n_1,n_2,\ldots}:(n_1,n_2,\ldots)\in \mathbb N^{\mathbb N}\}$are equi-uniformly $\mathcal T_{bx}-$continuous.
\newline (ii) The map $F:{\mathbb N}^{\mathbb N}\times \Sigma_{\infty}(A_{\infty})\to[0,\infty)$ is lower semi-continuous for the product topology on $\mathbb N^{\mathbb N}\times (\Sigma_{\infty}(A_{\infty}),\mathcal T_p).$

Let $$ f(u_1,u_2,\ldots)=\inf_{(n_1,n_2,\ldots)\in \mathbb N^{\mathbb N}}F(n_1,n_2,\ldots;u_1,u_2,\ldots).$$

If $f'_V(\emptyset)=0$ for every block subspace $V$, then
given $\epsilon>0$ there is a block subspace $V$ so that $\{f<\epsilon\}$ is $V$-strategically large.
\end{Thm}

\begin{proof}
For each $(n_1,\ldots,n_k)\in \mathbb N^{<\infty}$ we define $$f_{n_1,n_2,\ldots,n_k}(u_1,u_2,\ldots)=\inf_{m_1,m_2\ldots} F(n_1,\ldots,n_k,m_1,m_2,\ldots;u_1,u_2,\ldots).$$
The family $f_{n_1,n_2,\ldots,n_k}$ is $\mathcal T_{bx}$-equi-uniformly continuous.
By passing to a block subspace we can assume that each $f_{n_1,n_2,\ldots,n_k}$ is stable.  Of course the family $f'_{n_1,\ldots,n_k}$ is also  $\mathcal T_{bx}$-equi-uniformly continuous.
Let
$$ h_{n_1,\ldots,n_k}(u_1,\ldots,u_r)=\inf_{m\in\mathbb N}f'_{n_1,\ldots,n_k,m}(u_1,\ldots,u_r) \qquad (u_1,\ldots,u_r)\in\Sigma_{<\infty}(A_{\infty}).$$
This family is also equi-uniformly $\mathcal T_{bx}$-continuous.  Passing to a further block subspace we can suppose that this family is also stable.

By  Lemma \ref{sequences} we have that
$h'_{n_1,\ldots,n_k}\le f'_{n_1,\ldots,n_k}.$

Let us choose a sequence $(\epsilon_r)_{r=0}^{\infty}$ with $\sum\epsilon_r=\epsilon'<\epsilon.$
Again exploiting the countability of the family $h_{n_1,\ldots,n_k}$ we can pass to a further block subspace and, by relabelling as $E$, suppose that for each $(n_1,\ldots,n_k)\in \mathbb N^{<\infty}$ and $(u_1,\ldots,u_r)\in\Sigma_{<\infty}(A_{\infty})$  there is a strategy $\Phi_{n_1,\ldots,n_k,u_1,\ldots,u_r}$ with the property that
if $(V_j)_{j=1}^{\infty}$ is any sequence of subspaces then
for some $p\ge 1,$
$$ h_{n_1,\ldots,n_k}\Phi_{n_1,\ldots,n_k,u_1,\ldots,u_r}(u_1,\ldots,u_r,V_1,\ldots,V_p)<h'_{n_1,\ldots,n_k}(u_1,\ldots,u_r)+\epsilon_{k}.$$

We will now define a strategy $\Psi$.  To do this we first define  maps $\theta:\Sigma_{<\infty}(A_{\infty})\to \mathbb N^{<\infty}$ and $\varphi:\Sigma_{<\infty}(A_{\infty})\to \mathbb N$ such that $\varphi(u_1,\ldots,u_r)\le r.$
This is done inductively on the length of $(u_1,\ldots,u_r).$  We define
$\theta(\emptyset)=\emptyset$ and $\varphi(\emptyset)=0.$
Suppose that $\theta$ and $\varphi$ have been defined for all ranks up to $r$ and consider $(u_1,\ldots,u_{r+1})$.

 Let $\theta(u_1,\ldots,u_r)=(n_1,\ldots,n_k)$ and $\varphi(u_1,\ldots,u_r)=s.$ If $$h_{n_1,\ldots,n_k}(u_1,\ldots,u_{r+1})<h'_{n_1,\ldots,n_k}(u_1,\ldots,u_s)+\epsilon_k$$
then we can choose $m\in\mathbb N$ so that
$$ f'_{n_1,\ldots,n_k, m}(u_1,\ldots,u_{r+1})< h'_{n_1,\ldots,n_k}(u_1,\ldots,u_s)+\epsilon_k.$$  Let
$\theta(u_1,\ldots,u_{r+1})=(n_1,n_2,\ldots,n_k,m)$ and $\varphi(u_1,\ldots,u_{r+1})=r+1.$

Otherwise we simply put $\theta(u_1,\ldots,u_{r+1})=(n_1,\ldots,n_k)$ and $\varphi(u_1,\ldots,u_{r+1})=s.$

To define $\Psi$ we set
$$ \Psi(u_1,\ldots,u_r,V)=\Phi_{n_1,\ldots,n_k,u_1,\ldots,u_s}(u_1,\ldots,u_r,V)$$ where $(n_1,\ldots,n_k)=\theta(u_1,\ldots,u_r)$ and $s=\varphi(u_1,\ldots,u_r).$

Finally, we must show that if $(V_j)_{j=1}^{\infty}$ is any sequence of subspaces the sequence
$(u_1,u_2,\ldots)=\Psi(\emptyset,V_1,V_2,\ldots)$ is in $\{f<\epsilon\}.$

Let $k(r)$ be the length of $\theta(u_1,\ldots,u_r).$  Then $k(r)\le r$ for all $r.$
Suppose $k(r)$ remains bounded.  Then there exists $s$ so that $\varphi(u_1,\ldots,u_r)=s$ for all $r\ge s$ and $\theta(u_1,\ldots,u_k)=(n_1,\ldots,n_t)$ for some fixed $(n_1,n_2,\ldots,n_t)\in N^{<\infty}$ when $k\ge s.$
Thus $$(u_1,\ldots,u_r)=\Phi_{n_1,\ldots,n_t,u_1,\ldots,u_s}(u_1,\ldots,u_s,V_{s+1},\ldots,V_r) \qquad r\ge s.$$
It follows that for some $r>s$ we have
$$ h_{n_1,\ldots,n_t}(u_1,\ldots,u_r) < h'_{n_1,\ldots,n_t}(u_1,\ldots,u_s)+\epsilon_t$$ which implies that
$ \varphi(r)=r$ which is a contradiction.  Hence $k(r)\uparrow\infty.$
Let $r_j$ be the first natural number at which $k(r)=j.$
Then there exists $(n_1,n_2,\ldots)\in\mathbb N^{\mathbb N}$  so that
$$ \theta(u_1,\ldots,u_{r_j})=(n_1,\ldots,n_j).$$
By construction
$$ f_{n_1}'(u_1,\ldots,u_{r_1}) < h'(\emptyset)+\epsilon_0$$ and then for $j\ge 1$,
$$ f'_{n_1,\ldots,n_{j+1}}(u_1,\ldots,u_{r_{j+1}})< h'_{n_1,\ldots,n_j}(u_1,\ldots,u_{r_j})+\epsilon_j.$$
Since $h'_{n_1,\ldots,n_j}\le f'_{n_1,\ldots,n_j}$ we conclude that
$$ f'_{n_1,\ldots,n_j}(u_1,\ldots,u_{k_j})< \epsilon' $$ for all $j.$    But this implies the existence of $(n_{j,1},n_{j,2}\ldots)\in\mathbb N^{\mathbb N}$ and $(u_{j,1},u_{j,2},\ldots)\in\Sigma_{\infty}(A_{\infty})$ so that $n_{j,i}=n_i$ for $i\le j$ and $u_{j,i}=u_i$ for $i\le k_j$ and
$$ F(n_{j,1},n_{j,2},\ldots;u_{j,1},u_{j,2},\ldots)<\epsilon' \qquad j=1,2,\ldots.$$
Finally we invoke lower semi-continuity:
$$ F(n_1,n_2,\ldots,;u_1,u_2,\ldots)<\epsilon$$ and so $f(u_1,u_2,\ldots,)<\epsilon.$
\end{proof}

We now recall (Lemma \ref{Polish}) that $\Sigma_{\infty}(E)$ is a Polish space for the topology $\mathcal T_p.$  Thus every Borel set is analytic (i.e. a continuous image of $\mathbb N^{\mathbb N})$.

\begin{Thm}\label{main}  Let $\sigma$ be a large subset of $\Sigma_{\infty}(E).$  Suppose:
\newline (a) There is a sequence of  absolutely convex sets $C_n$ such that $C_n\cap F$ is compact for all finite-dimensional subspaces $F$ and $\sigma\subset \prod_{n=1}^{\infty}C_n.$
\newline (b) $\sigma$ is analytic as a subset of $(\Sigma_{\infty}(E),\mathcal T_p).$

Let $\rho_n$ be any sequence of $F$-norms on $E$ and define for $u=(u_1,u_2,\ldots),\ v=(v_1,v_2,\ldots)\in\Sigma_{\infty}(E)$
$$ d(u,v)=\sum_{j=1}^{\infty}\rho_j(u_j-v_j).$$

Let $\sigma_{\epsilon}=\{u=(u_j)_{j=0}^{\infty}:\ d(u,\sigma)=\inf_{v\in\sigma}d(u,v)<\epsilon\}.$  Then for every $\epsilon>0$ there is a block subspace $V$ so that $\sigma_{\epsilon}$ is strategically large for $V$.
\end{Thm}

\begin{proof}

We start by reducing this to the case when $C_n=\{x:\|x\|_{\infty}\le 1\}.$
To do this first observe that each $C_n$ is $\mathcal T-$closed.  Since $\sigma$ is large the linear space generated by $C_n$ is of finite codimension; if $E_n$ is a complementary space we can replace $C_n$ by the bigger set $C_n+K_n$ where $K_n$ is a compact absolutely convex neighborhood of the origin in $E_n.$  So we can suppose $C_n$ is absorbent and hence generates a norm $\|\cdot\|_n$ on $E$.  By induction, we can find a sequence of positive numbers $\delta_n$ so that $\|x\|=\sum_{n=1}^{\infty}\delta_n\|x\|_n<\infty$ for all $x\in E.$  Thus we can assume that each $C_n=\{x:\|x\|\le M_n\}$ for a single norm $\|\cdot\|$.

We can now pass to block basis which is a normalized  basic sequence in the completion of $(E,\|\cdot\|).$  Intersecting $\sigma$ with $\Sigma_{\infty}(V)$ for a block subspace gives again an analytic set since $\Sigma_{\infty}(V)$ is closed; thus we can relabel so that the block subspace is already $E$.   It now follows that each $C_n$ is included in a set $\{x:\|x\|_{\infty}\le M'_n\}$ where $M'_n$ is some sequence of positive numbers.  Finally we put
$\sigma'=\{(u_1,u_2,\ldots): (M'_1u_1,M'_2u_2,\ldots)\in\sigma\}$ and note that $\sigma'\subset \Sigma_{\infty}(A_{\infty}).$  Clearly it is enough to prove the result for $\sigma'$
with $\rho_j$ replaced by $\rho'_j(x)=\rho_j(M'_jx).$

We therefore assume that $\sigma\subset \Sigma_{\infty}(A_{\infty}).$

Now there is a continuous surjective map  $g:\mathbb N^{\mathbb N}\to \sigma$
for the $\mathcal T_p-$topology.  We will define
$$ F:\mathbb N^{\mathbb N}\times \Sigma_{\infty}(A_{\infty}) \to [0,\infty)$$ by
$$ F(n_1,n_2,\ldots; u_1,u_2,\ldots)=\min \Big(1, d((u_1,u_2,\ldots),g(n_1,n_2,\ldots))\Big).$$
It is clear that the family $f_{n_1,n_2,\ldots,}$ given by
$$ f_{n_1,n_2,\ldots}(u_1,u_2,\ldots)=F(n_1,n_2,\ldots; u_1,u_2,\ldots)$$ is equi-uniformly $\mathcal T_{bx}$-continuous.
It is also clear that $F$ is lower semi-continuous for the $\mathcal T_p$-topology in the second factor.

The result now follows directly from Theorem \ref{main0}.\end{proof}

\section{Applications to F-spaces}\label{applications}

We now apply these results to obtain the Gowers dichotomy for F-spaces.  Before doing this we make some remarks on  basic sequences in F-spaces.  There is an F-space (indeed a quasi-Banach space) which contains no basic sequence \cite{Kalton1995a}. It turns out that there is a dichotomy result for the existence of basic sequences with a very similar flavor to that of the Gowers dichotomy, which has been known for some time.

We will need some background (see \cite{KaltonPeckRoberts1984}).  Let $X$ be an F-space and let $\rho$ be an F-norm inducing the topology.  A basic sequence $(x_n)_{n=1}^{\infty}$ is called {\it regular} if $\inf_n\rho(x_n)>0.$  We denote by $\omega$ the space of all sequences (i.e. the countable product of lines).  The canonical basis of $\omega$ is not regular, and $\omega$ contains no regular basic sequence.  The following Lemma is elementary.

\begin{Prop}\label{regular} Suppose $X$ contains no subspace isomorphic to $\omega.$  Then given a basic sequence $(x_n)_{n=1}^{\infty}$ we may choose $a_n>0$ so that $(a_nx_n)_{n=1}^{\infty}$ is regular.\end{Prop}

\begin{proof} Indeed if not we have $\inf_{n\in\mathbb N}\sup_{t\in\mathbb R}\rho(te_n)=0.$  Then some subsequence of $(e_n)_{n=1}^{\infty}$ is equivalent to the canonical basis of $\omega.$ \end{proof}

Two subspaces $Y$, $Z$ of an F-space $X$ are called separated if $Y \cap Z = \{ 0\}$ and the canonical projection $Y+Z \to Y$ is continuous. An F-space is called HI if no
two infinite dimensional subspaces are separated.

\begin{Prop}\label{regularHI} Suppose $X$ has a regular basis $(e_n)_{n=1}^{\infty}$.  If there exist two separated  infinite-dimensional subspaces $Y,Z$ of $X$ then there exist two separated block subspaces of $X$.\end{Prop}

\begin{proof}  Since $X$ is regular the seminorm $\|x\|_\infty=\sup|e_n^*(x)|$ defines a continuous norm on $X.$  Now, fixing $0<\epsilon<1/8$ we may inductively define  a sequence $(x_n)_{n=1}^{\infty}\in X$ and  a block basic sequence $(u_n)_{n=1}^{\infty}$ such that:\newline
(i) $\|x_n\|_\infty=1$ for all $n,$ \newline (ii) $\rho(x_n-u_n)+\|x_n-u_n\|_\infty<\epsilon/2^n$ for all $n$ and \newline (iii) $x_n\in Y$ for $n$ odd, $x_n\in Z$ for $n$ even.

Note that $\|u_n\|_\infty \ge 1-\epsilon>1/2.$

Let $P_Y$ be the canonical projections of $Y+Z$ onto $Y$ with kernel $Z.$
For $v=\sum_{j=1}^na_ju_j$ in the linear span of the sequence $(u_n)_{n=1}^{\infty}$ we define
$Kv = \sum_{j=1}^n a_j(x_j-u_j).$  Then
$$\rho(Kv) \le \epsilon \max_{1\le j\le n}|a_j|\le 2\epsilon \|v\|_\infty$$ so that $K$ is continuous.   Furthermore
$$ \|Kv\|_\infty \le 2 \epsilon \max_{1\le j\le n} |a_j|\sup\|u_n\|_\infty\le 4\epsilon\|v\|_\infty.$$

 Thus $T=I+K$ extends to a continuous operator $T:[u_n]_{n=1}^{\infty}\to Y+Z$.   Now $$\|Tv\|_\infty \ge (1-4\epsilon)\|v\|_\infty\ge 1/2\|v\|_\infty, \qquad v\in [u_n]_{n=1}^{\infty}.$$  If $(v_n)_{n=1}^{\infty}$ is a sequence such that $\lim_{n\to\infty}\rho(Tv_n)=0$ then $\lim_{n\to\infty}\|Tv_n\|_\infty=0$ and so $\lim_{n\to\infty}\|v_n\|_\infty=0.$  Thus $\lim_{n\to\infty}\rho(Kv_n)=0$ and hence $\lim_{n\to\infty}\rho(v_n)=0.$  Hence $T$ is an isomorphism of $[u_n]_{n=1}^{\infty}$ into $Y+Z$.   Consider the operator $S=T^{-1}P_YT$: then $S$ is a projection of $[u_n]_{n=1}^{\infty}$ onto $[u_{2n-1}]_{n=1}^{\infty}$ and the block subspaces $V=[u_{2n-1}]_{n=1}^{\infty}$ and $W=[u_{2n}]_{n=1}^{\infty}$ are separated.\end{proof}

\begin{Thm}\label{GDuncbasic} Let $X$ be an F-space with a regular basis containing no unconditional basic sequence.  Then $X$ has an HI subspace $Y$.\end{Thm}

\begin{proof} We assume that  $X$ has a regular basis  $(e_n)_{n=1}^{\infty}.$

We now consider the countable dimensional $E$ with Hamel basis $(e_n)_{n=1}^{\infty}.$  Note that the norm $\|\cdot\|_\infty$ on $E$ is continuous with respect to the F-space topology since $(e_n)_{n=1}^{\infty}$ is regular. For any block basic sequence $(u_n)_{n=1}^{\infty}$ we say that $(u_n)_{n=1}^{\infty}$ is somewhat unconditional if the map
$$\sum_{j=1}^\infty a_ju_j\to \sum_{j=1}^\infty(-1)^ja_ju_j$$ (defined for $(a_j)_{j=1}^{\infty}\in c_{00}$) is continuous for the F-space topology restricted to $E.$
Let $\sigma_0$ be the collection of all somewhat unconditional sequences.  We claim that with respect to $\mathcal T_p$ this set is a Borel subset of $\Sigma_\infty(E).$  Indeed let $(U_m)_{m=1}^{\infty}$ be a base of open neighborhoods of zero.  Let $\sigma_0(m,n)$ be the set of $(u_j)_{j=1}^{\infty}$ so that
$$ \sum_{j=1}^\infty a_ju_j\in U_m \ \implies\ \sum_{j=1}^\infty(-1)^ja_ju_j\in \overline{U}_n.$$
Then $\sigma_0(m,n)$ is $\mathcal T_p-$closed and $\sigma_0 = \cap_{n=1}^{\infty}\cup_{m=1}^{\infty}\sigma_0(m,n).$

We define $\sigma$ to be the subset of $\Sigma_\infty{E}$ of all block basic sequences $(u_n)_{n=1}^{\infty}$ such that $\|u_n\|_\infty=1$ for all $n$ and $(u_n)_{n=1}^{\infty}$ fails to be somewhat unconditional.  Then $\sigma$ is also Borel in $\mathcal T_p.$  Furthermore, since $X$ contains no unconditional basic sequence we conclude that $\sigma$ is large.

Fix some $0<\epsilon<1$ and let $\sigma'$ be the subset of $\Sigma_\infty(E)$ of all sequences $(v_j)_{j=1}^{\infty}$ such that
$$ \inf\{\sum_{j=1}^{\infty}(\|u_j-v_j\|_\infty+\rho(u_j-v_j)) :\ (u_j)_{j=1}^{\infty}\in\sigma\}<\epsilon.$$
Note that if $(v_j)_{j=1}^{\infty}\in \sigma'$ there exists $(u_j)_{j=1}^{\infty}\in\sigma$ which is equivalent to $(v_j)_{j=1}^{\infty}.$  Hence each $(v_j)_{j=1}^{\infty}\in\sigma'$ fails to be somewhat unconditional.

 According to Theorem \ref{main} we can find a block subspace $V$ so that $\sigma'$ is strategically large for $V$. Let $Y$ be the closure of $V$.  We show that $Y$ is HI.  $Y$ has a regular basis $(u_n)_{n=1}^{\infty}$ which is a block basis of $(e_n)_{n=1}^{\infty}.$   According to Proposition \ref{regularHI} we need only check that if $W_1,W_2$ are two block subspaces of $V$ then $W_1$ and $W_2$ cannot be separated.

 Let $\Phi$ be the strategy guaranteed by the fact that $\sigma'$ is strategically large.  Then
 $\Phi(W_1,W_2,W_1,W_2,\ldots)=(v_1,v_2,\ldots)$ and the sequence $(v_j)_{j=1}^{\infty}$ fails to be somewhat unconditional so that $W_1,W_2$ are not separated.
 \end{proof}

Let us now recall the criterion of the existence of basic sequences given in \cite{Kalton1974a} (see also \cite{KaltonShapiro1976}). An F-space $X$ is called {\it minimal} if  there is no strictly weaker Hausdorff topology on $X$.

\begin{Prop}\label{existencebasicsequences} If $X$ is a non-minimal F-space then $X$ contains a regular basic sequence.\end{Prop}

Let us call an infinite-dimensional F-space $X$ {\it strongly HI (SHI)} if it contains a non-zero vector $e$ so that $e\in L$ for {\it every} infinite-dimensional closed subspace $L$ of $X$.   We remark that it is possible to consider spaces $X$ which satisfy the slightly stronger condition that any two infinite-dimensional closed subspaces have non-trivial intersection; this condition implies $X$ contains no basic sequence, but it is not clear if it implies that $X$ is SHI.
The problem is that we do not know if, under this condition, the intersection of any three infinite-dimensional closed subspaces is non-trivial.  This is related to the fact, discussed later, that the sum of two strictly singular operators need not be strictly singular (see the discussion after Theorem \ref{ss}).

\begin{Thm}\label{GDbasic} Let $X$ be an F-space containing no basic sequence.  Then $X$ has an SHI-subspace $Y$. \end{Thm}

\begin{proof}  We may assume that $X$ is separable.  Let us say that a collection $\mathcal L$ of closed subspaces of $X$ is a {\it subspace-filter} in $X$ if each $L\in\mathcal L$ is infinite-dimensional and $L_1\cap L_2\in\mathcal L$ whenever $L_1,L_2\in\mathcal L$; we say that a subspace-filter $\mathcal L$ is a {\it subspace-ultrafilter} if it is not contained properly in any other subspace-filter. We then pick $\mathcal L$ to be a subspace-filter such that $H=\cap\{L:L\in\mathcal L\}$ has minimal dimension ($1\le \dim H\le \infty$).

We will argue that $\dim H>0.$ Indeed if $H=\{0\}$ then we define a topology $\tau$ on $X$ by taking as a base of neighborhoods
sets of the form $U+L$ where $U$ is a neighborhood of zero in the F-space topology and $L\in\mathcal L.$ If $H=\{0\}$ then $\tau$ is Hausdorff. By Proposition \ref{existencebasicsequences} we have that $\tau$ coincides with the original topology.  Then we may find a strictly decreasing sequence $L_n\in\mathcal L$ so that $L_n\subset\{x:\rho(x)<2^{-n}\}.$  If we pick $x_n\in L_n\setminus L_{n+1},$ it is easy to verify that $(x_n)_{n=1}^{\infty}$ is a basic sequence equivalent to the canonical basis of $\omega.$

If $\dim H=\infty$ then it follows from maximality that $H$ has no proper closed infinite-dimensional subspace and so we may take $Y=H$ and $y$ any non-zero element of $Y.$  If $\dim H<\infty$ we first argue by Lindelof's theorem that since $X$ is separable we can find a descending sequence of subspaces $L_n\in\mathcal L$ so that $\cap L_n=H.$  We may suppose this sequence is strictly descending and take $x_n\in L_n\setminus L_{n+1}$ for $n\ge 1.$  Let $V_n=[x_k]_{k\ge n}$ so that $V_n\subset L_n.$  Suppose $W$ is any closed infinite-dimensional subspace of $V_1;$ then $\dim V_n\cap W=\infty$ for each $n.$ Let $\mathcal L'$ be any subspace-ultrafilter containing each $V_n$ and $W$.  Then $\cap\{L:L\in\mathcal L'\}\subset H$ but the inclusion cannot be strict because the original minimality assumption on $\dim H.$  Hence $H\subset W.$  Thus we can take $Y=V_1$ and $y\in H\setminus\{0\}.$\end{proof}

An examination of the proof shows that we have actually proved a slightly stronger result:

\begin{Cor}\label{GDbasic2} Let $X$ be an F-space containing no basic sequence.  Then $X$ has an SHI-subspace $Y$ with the property that if $E$ is the intersection of all infinite-dimensional subspaces of $Y$ then there is a descending sequence of infinite-dimensional subspaces $(L_n)_{n=1}^{\infty}$ of $Y$ with $\cap_{n=1}^{\infty}L_n=E.$ \end{Cor}

We are now ready to establish the full force of the Gowers dichotomy for F-spaces.

\begin{Thm}\label{Gowers}  Let $X$ be an F-space.  If $X$ contains no unconditional basic sequence, then $X$ contains an HI subspace.\end{Thm}

\begin{proof}  If $X$ contains no basic sequence then $X$ contains a SHI subspace (Theorem \ref{GDbasic}).  So we may assume $X$ has a basis.  Clearly $X$ cannot contain a copy of $\omega$ so we can assume the basis is regular (Proposition \ref{regular}).  Now apply Theorem \ref{GDuncbasic}.\end{proof}

We conclude this section with:

\begin{Thm}\label{Gowers2} Let $X$ be an HI F-space.  Suppose $X$ has a closed infinite-dimensional subspace containing no basic sequence.  Then $X$ contains no basic sequence.\end{Thm}

\begin{proof} We will show that if $(V_n)_{n=1}^{\infty}$ is any descending sequence of closed infinite-dimensional subspaces of $X$ then $\cap_{n=1}^{\infty}V_n\neq \{0\}.$
We use Corollary \ref{GDbasic2} to deduce the existence of a descending sequence of infinite-dimensional closed subspaces $(L_n)_{n=1}^{\infty}$ such that, if $E=\cap_{n=1}^{\infty}L_n$, then $E\neq\{0\}$ and $E$ is contained in any infinite-dimensional subspace of $L_1.$  Consider the sequence $(L_n\cap V_n)_{n=1}^{\infty}.$  Then if $\dim L_n\cap V_n=\infty$ for all $n$ we have $E\subset \cap_{n=1}^{\infty}L_n\cap V_n\subset \cap_{n=1}^{\infty}V_n.$

If not then there exists $n_0$ such that $\dim (L_n\cap V_n)$ is finite and constant for $n\ge n_0.$  Hence $L_n\cap V_n =F$ some fixed finite dimensional subspace for $n\ge n_0.$  We show $\dim F>0.$  If for some $n\ge n_0$ we have $L_n\cap V_n=\{0\}$ then $L_n+V_n$ cannot be closed since $X$ is HI.  Thus there are sequences $(x_k)_{k=1}^{\infty}$ in $L_n$ and $(v_k)_{k=1}^{\infty}\in V_n$ so that $\lim\rho(x_k+v_k)=0$ but $\rho(x_k)\ge \delta>0$ for all $k.$  Consider the metric topology on $L_n$ defined by the F-norm $x\to d(x,V_n):=\inf\{\rho(x+v):\ v\in V\}.$  This is topology is Hausdorff on $L_n$ and strictly weaker than the $\rho-$topology.  Hence $L_n$ contains a basic sequence by Proposition \ref{existencebasicsequences}, and this is a contradiction.  Hence $\dim F>0$ and $F\subset \cap_{n=1}^{\infty}V_n.$\end{proof}

\section{Strictly singular maps}\label{strictly}

In \cite{GowersMaurey1993} the following Theorem is shown:

\begin{Thm}\label{ss}  Let $X$ be a complex Banach space.  If $X$ is HI then every bounded linear operator $T:X\to X$ is of the form $T=\lambda I +S$ where $S$ is strictly singular.\end{Thm}

We do not know whether such a theorem can hold for a complex F-space but we show that it holds equally for complex quasi-Banach spaces.  There are some small wrinkles in the proof as the reader will see.

>From now on we will deal with quasi-Banach space $X$ (or $Y$, etc.) with a given quasi-norm which is assumed to be $p$-subadditive (for a suitable $0<p\le 1$) i.e.
$$ \|x+y\|^p\le \|x\|^p+\|y\|^p, \qquad x,y\in X.$$

A linear operator $T:X\to Y$ is an {\it isomorphic embedding} if there exists $c>0$ so that
$\|Tx\|\ge c\|x\|$ for $x\in X.$  $T$ is called {\it strictly singular} if $T|_V$ fails to be an isomorphic embedding for every infinite-dimensional subspace $V$ of $X$.  $T$ is called {\it semi-Fredholm} if $\ker T$ is finite-dimensional and $T(X)$ is closed.  $T$ is called {\it Fredholm} if $T$ is semi-Fredholm and $\dim Y/T(X)<\infty.$

$T$ is semi-Fredholm if and only if for every bounded sequence $(x_n)_{n=1}^{\infty}$ such that $\lim_{n\to\infty}\|Tx_n\|=0$ we can extract a convergent subsequence.  Thus it is clear the restriction of a semi-Fredholm operator to an infinite-dimensional closed subspace remains semi-Fredholm.

Let us make some remarks.  Suppose $X$ is a SHI space and let $E_X$ be the intersection of all closed infinite-dimensional subspaces of $X.$  If $\dim E_X=\infty$ then $E_X$ is an {\it atomic space} i.e. it has no proper closed infinite-dimensional subspace.  The existence of atomic spaces is still open (the only known results in this direction are in \cite{Reese1992}).  However it is known that there exist quasi-Banach spaces $X$ for which $E_X$ is finite-dimensional and non-trivial, even with $\dim E_X>1$ (\cite{Kalton1995}, Theorem 5.5).  The quotient map $Q:X\to X/E_X$ is then both semi-Fredholm and strictly singular (this cannot happen for operators on Banach spaces).  Furthermore if $\dim E_X>1$ then  let $L_1,L_2$ be two distinct one-dimensional subspaces of $E_X$.  Then the quotient maps $Q_1:X\to X/L_1$ and $Q_2:X\to X/L_2$ are both strictly singular and semi-Fredholm.  However the map $x\to (Q_1x,Q_2x)$ from $X$ into $X/L_1\oplus X/L_2$ is an isomorphism.  Thus the sum of two strictly singular operators need not be strictly singular!

The key fact we will need is the following:

\begin{Thm}\label{semiF} Let $X$ be an infinite-dimensional complex quasi-Banach space and suppose $T:X\to X$ is a bounded operator.  Then there exists $\lambda\in\mathbb C$ so that $T-\lambda I$ is not semi-Fredholm.\end{Thm}

This Theorem is proved for Banach spaces by Gowers and Maurey \cite{GowersMaurey1993}.  The proof for quasi-Banach spaces requires some additional tricks.  These tricks are necessitated by the fact that finite-dimensional subspaces are not always complemented.

We list the relevant facts we need:

\begin{Prop}\label{spectralradiusformula}  If $X$ is a complex quasi-Banach space and $T:X\to X$ is a bounded linear operator then $\text{Sp}(T)=\{\lambda\in\mathbb C:\ T-\lambda I\ \text{is not invertible}\}$ is a non-empty compact set and $\max_{\lambda\in\Sp(T)}|\lambda|=\lim_{n\to\infty}\|T^n\|^{1/n}.$\end{Prop}

This is due to Zelazko \cite{Zelazko1961}.  We point out that the key ideas in the proof involve a reduction to the Banach algebra case.  One starts with the fact (\cite{Zelazko1961}) that on a commutative quasi-Banach algebra the formula $r(x)=\lim_{n\to\infty}\|T^n\|^{1/n}$ defines a seminorm. Using this one can prove the Gelfand-Mazur theorem (see e.g. \cite{KaltonPeckRoberts1984}) in this context and develop the basic theory of commutative quasi-Banach algebras.  The Proposition is obtained by looking at the double commutant of $T.$

\begin{Prop}\label{clopen}  Let $X$ be a complex quasi-Banach space and let $\mathcal G_1$ denote the subset of $\mathcal L(X)$ consisting of all isomorphic embeddings and $\mathcal G_2$ be the collection of all surjections.  Then $\mathcal G_1$ and $\mathcal G_2$ are both open sets and $\mathcal G_1\cap\mathcal G_2$ is a clopen subset relative to $\mathcal G_1$ and relative to $\mathcal G_2.$\end{Prop}

See \cite{KaltonPeckRoberts1984} pp. 132-134.

\begin{Prop}\label{GowersMaurey} Let $X$ be an infinite-dimensional complex Banach space and suppose $T:X\to X$ is quasi-nilpotent, i.e. $\text{Sp}(T)=\{0\}.$  Then $T$ cannot be semi-Fredholm.\end{Prop}

See \cite{GowersMaurey1993} Lemma 19.  We will now need to prove this Proposition for a general complex quasi-Banach space.  We do this is in several very simple steps.  Assume throughout that $X$ is an infinite-dimensional complex quasi-Banach space.

\begin{Lem}\label{000}  Suppose $T:X\to X$ is any bounded operator and $\lambda\in\partial\text{Sp}(T).$  Then $T-\lambda I$ can be neither an isomorphic embedding nor a surjection.\end{Lem}

\begin{proof} This follows from Proposition \ref{clopen}.\end{proof}

\begin{Lem}\label{111} Suppose $X$ has trivial dual.  If $T:X\to X$ is quasi-nilpotent then $T$ cannot be Fredholm.\end{Lem}

\begin{proof} If $T(X)$ has finite codimension in $X$ then $T$ is onto in this case.  We then use Lemma \ref{000}.\end{proof}

\begin{Lem} \label{222} If $X$ is any infinite-dimensional complex quasi-Banach space and $T:X\to X$ is quasi-nilpotent  then $T$ cannot be Fredholm.\end{Lem}

\begin{proof}  Denote by $X^*$ the dual of $X$; this is a Banach space but it can be quite small (even $\{0\}$).  We assume $X^*\neq\{0\}$ as this case is covered in Lemma \ref{111}. Assume $T:X\to X$ is quasi-nilpotent and Fredholm. Then $T^*:X^*\to X^*$ is Fredholm. In fact $T^*(X^*)=\ker (T)^{\perp}$;  this depends on the fact that every continuous linear functional $y^*$ on $T(X)$ can be extended to $x^*\in X^*$ since $\dim X/T(X)<\infty.$ Since $\|(T^*)^n\|\le \|T^n\|$ the spectral radius formula shows that $T^*$ is quasi-nilpotent.  By Proposition \ref{GowersMaurey} we must have $\dim X^*<\infty.$  Let $X_0=\{x\in X:\ x^*(x)=0\ \forall\ x^*\in X^* \}.$   Then $X_0$ is invariant for $T$ and of finite-codimension in $X$.  Clearly $X_0^*=\{0\}$ and $T|_{X_0\to X_0}$ remains Fredholm so we can apply Lemma \ref{111} to get a contradiction.\end{proof}

\begin{Lem}\label{333}  If $X$ is any infinite-dimensional complex quasi-Banach space and $T:X\to X$ is quasi-nilpotent  then $T$ cannot be semi-Fredholm.\end{Lem}

\begin{proof}  Assume $T$ is semi-Fredholm.  Then by a Baire Category argument there exists $x\in X$ so that $T^nx\neq 0$ for every $n \in {\mathbb N}.$  Let $Y=[T^nx]_{n=1}^{\infty}.$  Then $T:Y\to Y$ is Fredholm and remains quasi-nilpotent (using Proposition \ref{spectralradiusformula}).  Clearly $T|_{Y\to Y}$ is not nilpotent so $\dim Y=\infty.$  This is a contradiction by Lemma \ref{222}.
\end{proof}

\begin{proof}[Proof of Theorem \ref{semiF}]  The remaining steps in the proof of Theorem \ref{semiF} are very similar to those in \cite{GowersMaurey1993} for the Banach space case. Assume $T-\lambda I$ is semi-Fredholm for all $\lambda\in\mathbb C.$ We suppose $\lambda\in\partial \text{Sp}(T)$ is an accumulation point of $\partial\text{Sp}(T).$  Let $\lambda_n\to\lambda$ with $\lambda_n\ne\lambda$ and $\lambda_n\in\partial\text{Sp}(T).$  Each $\lambda_n$ is an eigenvalue of $T$ (by Lemma \ref{000}, since $T-\lambda_nI$ is semi-Fredholm), say with eigenvector $x_n.$   Let $Y=[x_n]_{n=1}^{\infty}.$  Then $Y$ is invariant for $T$ and $\lambda\in \partial\text{Sp}(T|_{Y\to Y}).$  However $(T-\lambda I)|_{Y\to Y}$ has dense range and is semi-Fredholm.  Hence $(T-\lambda I)|_{Y\to Y}$ is surjective and we have a contradiction by Lemma \ref{000}.

It follows that $\partial\text{Sp}(T)$ has no accumulation points and hence is a finite set.  Thus $\text{Sp}(T)$ is also finite say $\text{Sp(T)}=\{\lambda_1,\ldots,\lambda_n\}.$  Then $S=\prod_{k=1}^n(T-\lambda_kI)$ is semi-Fredholm and $\text{Sp}(S)=\{0\}.$  This contradicts Lemma \ref{333}.\end{proof}

\begin{Thm} \label{ss1} Let $X$ be an infinite-dimensional complex quasi-Banach space. If $T:X\to X$ is strictly singular then $T$ cannot be semi-Fredholm.\end{Thm}

{\it Remark.} Note that this is false for operators $T:X\to Y$ by the remarks above.

\begin{proof}  In fact $T-\lambda I$ is always semi-Fredholm if $\lambda\neq 0$ (Theorem 7.10 of \cite{KaltonPeckRoberts1984}).  The result follows from Theorem \ref{semiF}.\end{proof}

\begin{Thm}\label{semiFss}  Let $X$ be an infinite-dimensional complex quasi-Banach space.  If $T:X\to X$ is a bounded linear operator then exactly one of the following two conditions holds:
\newline (i) For every $\epsilon>0$ there is an infinite-dimensional closed subspace $V$ of $X$ such that $\|T|_V\|<\epsilon.$
\newline (ii) $T$ is semi-Fredholm.

If $X$ is HI then (i) is equivalent to
\newline (i)$^{\prime}$ $T$ is strictly singular.
\end{Thm}

\begin{proof}  Assume (ii).  Then there is a constant $c>0$ so that $\|Tx\|\ge cd(x,F)$ for $x\in X,$ where $F=\ker T.$ If $V$ is an infinite-dimensional closed subspace we can find a sequence $(v_n)_{n=1}^{\infty}$ in the unit ball so
$ \|v_m-v_n\|\ge 1/2$ for $m\neq n.$  Assuming that the norm is $p-$convex, by a simple compactness argument we can then show the existence of a pair $m\neq n$ so that $ (d(v_m,F)^p+d(v_n,F)^p)^{1/p}\ge 1/4.$   Hence $\|Tv_m\|^p+\|Tv_n\|^p \ge (1/4)^p.c^p$  This implies a lower bound on $\|T|_V\|.$

Now assume (ii) fails and that $F=\text{ker}(T)$ is finite dimensional.  Then $T$ factors in the form $T=T_0Q$ where $Q:X\to X/F$ is the quotient map and $T_0:X/F\to X$ is one-one but not an isomorphic embedding.  Then there is a normalized sequence $\xi_n\in X/F$ so that $\|T_0\xi_n\|<2^{-n}.$  Now using Theorem 4.6 of \cite{KaltonPeckRoberts1984} we can assume by passing to a subsequence that $(\xi_n)_{n=1}^{\infty}$ satisfies an estimate
$$ \max_{1\le k\le n}|a_k|\le C\|\sum_{k=1}^n a_k\xi_k\|, \qquad a_1,\ldots,a_n\in\mathbb C.$$
In particular if $V_k=Q^{-1}[\xi_j]_{j\ge k}$ then each $V_k$ is infinite-dimensional and $\|T|_{V_k}\|\to 0.$  Thus (i) holds.

Now assume $X$ is HI.  Suppose $T$ satisfies (i) and is not strictly singular.  Then there is an infinite-dimensional subspace $W$ so that $\|Tw\|\ge \delta\|w\|$ for $w\in W$ where $\delta>0.$  Pick $\epsilon=\delta/2$ and then choose $V$ as in (i) for this $\epsilon.$  Clearly $V\cap W=\{0\}.$

Now assume $v\in V,w\in W$ with $\|v+w\|=1.$  Then
\begin{align*}
\|v\|^p&\le 1+\|w\|^p\\
&\le 1+2^{-p}\|v\|^p+\|w\|^p-2^{-p}\|v\|^p\\
&\le 1+2^{-p}\|v\|^p+ \delta^{-p}\|Tw\|^p-2^{-p}\epsilon^{-p}\|Tv\|^p\\
&= 1+2^{-p}\|v\|^p +\delta^{-p}(\|Tw\|^p-\|Tv\|^p)\\
&\le 1+2^{-p}\|v\|^p +\delta^{-p}\|T(v+w)\|^p\\
&\le 1+2^{-p}\|v\|^p +\delta^{-p}\|T\|^p.\end{align*}  Thus
$$ \|v\| \le \left(\frac{1+\delta^{-p}\|T\|^p}{1-2^{-p}}\right)^{1/p}.$$
This contradicts the fact that $X$ is HI.

Conversely if $T$ is strictly singular it cannot be semi-Fredholm by Theorem \ref{ss1} and so (i) must hold.
\end{proof}

 \begin{Thm}\label{ss2}  Let $X$ be an infinite-dimensional complex quasi-Banach space.  If $X$ is HI then every bounded linear operator $T:X\to X$ is of the form $T=\lambda I +S$ where $S$ is strictly singular.\end{Thm}

 \begin{proof} There exists $\lambda$ so that $T-\lambda I$ is not semi-Fredholm by Theorem \ref{semiF}.  By Theorem \ref{semiFss} this means $T-\lambda I$ is strictly singular.
 \end{proof}

 In the case when $X$ is SHI this result is much simpler.  Indeed we have:

 \begin{Thm}\label{SHI}  Let $X$ be an SHI space and suppose $E$ is the intersection of all infinite-dimensional subspaces of $X$. Let $Q:X\to X/E$ be the quotient map (which is strictly singular). Then if $T:X\to X$ is a bounded operator, there exists $\lambda\in\mathbb C$ and a bounded operator $S:X/E\to X$ so that $T=\lambda I+SQ.$\end{Thm}

 \begin{proof}  Let us first give a simpler proof of Theorem \ref{ss2}.  It is clearly that if $R:X\to X$ is an invertible operator then $R(E)\subset E$ and this implies that $E$ is invariant for all operators on $X$.  If $E$ is atomic then $E$ is rigid (\cite{KaltonPeckRoberts1984} Theorem 7.22, p. 155).  Otherwise $E$ is finite-dimensional.  In either case $T|_E$ has an eigenvalue $\lambda$ and so $T-\lambda I$ factors through a quotient map $Q':X\to X/F'$ where $F'$ is a non-trivial subspace of $E$.  Hence $T-\lambda I$ is strictly singular.

Now using Theorem \ref{semiFss} it is clear any strictly singular operator on $X$ vanishes on $E$ and so we get the desired factorization.\end{proof}

\end{document}